\documentclass[12pt]{article}
\usepackage{amssymb,amsmath}
\usepackage{cases}
\usepackage{amsthm}
\usepackage{amsfonts}
\usepackage{graphicx}
\usepackage{cite,color,xcolor}
\usepackage[left=2.3cm,right=2.3cm,top=2.3cm,bottom=2.3cm]{geometry}
\usepackage[colorlinks,citecolor=blue,urlcolor=blue]{hyperref}
\usepackage[utf8]{inputenc}

\newtheorem{theorem}{Theorem}[section]

\newtheorem{lemma}{Lemma}[section]

\newtheorem{definition}{Definition}[section]
\newtheorem{remark}{Remark}[section]

\usepackage{txfonts}

\newcommand{\bal}{\begin{align}}
\newcommand{\bbal}{\begin{align*}}
\newcommand{\beq}{\begin{equation}}
\newcommand{\eeq}{\end{equation}}
\newcommand{\bca}{\begin{cases}}
\newcommand{\eca}{\end{cases}}

\newcommand{\pa}{\partial}
\newcommand{\fr}{\frac}

\newcommand{\De}{\Delta}

\newcommand{\cd}{\cdot}

\newcommand{\dd}{\mathrm{d}}

\newcommand{\R}{\mathbb{R}}

\newcommand\f{\left}
\newcommand\g{\right}
\begin{document}
\bibliographystyle{plain}
\title{Ill-posedness for a generalized Camassa-Holm equation with higher-order nonlinearity in the critical Besov space}

\author{Wei Deng$^{1}$, Min Li$^{2}$\footnote{E-mail: limin@jxufe.edu.cn (Corresponding author)},  Xing Wu$^{3}$ and Weipeng Zhu$^{4}$\\
\small $^1$ Department of Mathematics, Ganzhou Teachers College, Ganzhou 341000, China\\
\small $^2$ Department of Mathematics, Jiangxi University of Finance and Economics, Nanchang 330032, China\\
\small $^3$ College of Information and Management Science,
Henan Agricultural University, Zhengzhou 450002, China\\
\small $^4$ School of Mathematics and Big Data, Foshan University, Foshan, Guangdong 528000, China}

\date{\today}

\maketitle\noindent{\hrulefill}

{\bf Abstract:} In this paper, we prove that the Cauchy problem for a generalized Camassa-Holm equation with higher-order nonlinearity is ill-posed in the critical Besov space $B^1_{\infty,1}(\R)$. It is shown in (J. Differ. Equ., 327:127-144,2022) that the Camassa-Holm equation is ill-posed in $B^1_{\infty,1}(\R)$, here we turn our attention to a higher-order nonlinear generalization of Camassa-Holm equation proposed by Hakkaev and Kirchev (Commun Partial Differ Equ 30:761-781,2005).  With newly constructed initial data, we get the norm inflation in the critical space $B^1_{\infty,1}(\R)$ which leads to ill-posedness.


\vskip0mm\noindent{\hrulefill}

\section{Introduction}

In this paper, we are concerned with the Cauchy problem for the higher-order nonlinear generalized Camassa-Holm equation
\begin{eqnarray}\label{eq1}
        \left\{\begin{array}{ll}
         u_t-u_{xxt}+\frac{(Q+2)(Q+1)}{2}u^Qu_x=\big(\frac{Q}{2}u^{Q-1}u_x^2+u^Qu_{xx}\big)_x, ~~t>0, ~x\in \mathbb{R},\\
          u(0, x)=u_0, ~~x\in \mathbb{R},\end{array}\right.
        \end{eqnarray}
where $Q\geqslant 1$ is a positive integer and $u(t, x)$ stands for the fluid velocity at time $t\geqslant 0$ in the spatial direction. The equation (\ref{eq1}) was introduced by Hakkaev and Kirchev\cite{hk1} in a more general case
$$
u_t-u_{xxt}+(a(u))_x=\big(\frac{b'(u)}{2}u_x^2+b(u)u_{xx}\big)_x,
$$
take $a(u)=\frac{Q+2}{2}u^{Q+1},~b(u)=u^Q$ we get equation (\ref{eq1}). When $Q=1$, (\ref{eq1})  is reduced to the well-known Camassa-Holm equation
\begin{align}\label{che}\tag{CH}
u_t-u_{xxt}+3uu_x=2u_xu_{xx}+uu_{xxx}.
\end{align}
The Camassa-Holm equation (\ref{che})  was first founded by Fuchssteiner and Fokas\cite{ff} as a completely integrable generalization of the Korteweg-de-Vries (KdV) equation with bi-Hamiltonian structure, and was later recovered as a water wave model by Camassa and Holm\cite{ch} to describe the unidirectional propagation of shallow water waves over a flat bottom. Most importantly, CH equation has peakon solutions of the form $Ce^{-|x-Ct|}$ which aroused a lot of interest in physics, see \cite{c5,t}. There is an extensive literature about the strong well-posedness, weak solutions and analytic or geometric properties of the CH equation, here we name some. Local well-posedness and ill-posedness for the Cauchy problem of the CH equation were investigated in \cite{ce2,d2,glmy}. Blow-up phenomena and global existence of strong solutions were discussed in \cite{c2,ce2,ce3,ce4}. The existence of global weak solutions and dissipative solutions were investigated in \cite{bc1,bc2,xz1}, more results can be found in the references therein.

As mentioned by Hakkaev and Kirchev in \cite{hk2}, (\ref{eq1}) is a generalization of Camassa-Holm equation in the way that it preserves two important conservation laws of (\ref{che}), namely
\begin{align}\label{efp}
E(u)=\int_{\R}(u^2+u_x^2)dx,~~~~ F(u)=\int_{\R}(u^{Q+2}+u^Qu_x^2)dx.
  \end{align}
These are the key conserved quantities to explore the orbital stability and instability of solitary wave solutions of (\ref{che}). We are mainly concerned with the Cauchy problem of (\ref{eq1}), here is a brief review of the previous literature. Hakkafv and Kirchev \cite{hk1} proved that (\ref{eq1}) is locally well-posed with the initial data in $H^s(\R), s>\frac32$, and also obtained the stability of peakons and orbital stability of solitary wave solution of (\ref{eq1}). Yan et al. \cite{ylz} proved that the solutions to the Cauchy problem (\ref{eq1}) do not depend uniformly continuously on the initial data in $H^s(\R), s<\frac32$ and showed the local well-posedness in $B^{3/2}_{2,1}(\R)$.

Recently, a lot of literature was devoted to studying the well-posedness (especially the non-uniform dependence and ill-posedness) problem of the Camassa-Holm type equations in the critical Besov spaces\cite{hk,hkm,wyx}. For example, Guo et al. \cite{glmy} proved norm inflation and hence ill-posedness for the Camassa-Holm equation in the critical Sobolev space $H^{3/2}(\R)$ which solves the open problem left by Danchin \cite{d2}. In fact, in \cite{glmy} they proved the ill-posedness for the Camassa-Holm type equations in Besov space $B^{1+1/p}_{p,r}(\mathbb{R})$ with $(p,r)\in [1,\infty]\times(1,\infty]$ which means $B^{1+1/p}_{p,1}$ is the critical Besov space. Then, Li et al. \cite{lyz1,lyz2} demonstrated the non-continuity and sharp ill-posedness of the Camassa-Holm type equations in $B^s_{p,\infty}(\mathbb{R})$ with $s>\max\{\frac32,1+\frac1p\}$. Later, the local well-posedness for the Camassa-Holm type equations in $B^{1+1/p}_{p,1}(\mathbb{R})$ with $p\in[1,\infty)$ has been proved by Ye et al.\cite{yyg} through the compactness argument and Lagrangian coordinate transformation. In the remaining case $p=\infty$, Guo et al.\cite{gyy} proved the ill-posedness for the CH equation in $B^1_{\infty,1}(\R)$ by exhibiting the norm inflation, then Li et al. \cite{lyz3} has got the ill-posedness for the Novikov equation in the same space. Quite unexpectedly, Li et al. \cite{lyz4} have fund that the Degasperis-Procesi equatoin is well-posed in $B^1_{\infty,1}(\R)$. In this paper,  we will deduce the ill-posedness of the Cauchy problem (\ref{eq1}) in $B^1_{\infty,1}(\R)$. Due to the strong nonlinearity of the equation (\ref{eq1}), our construction of the initial data is quite different from the previous proof.

We first rewrite (\ref{eq1}) in the following equivalent nonlocal form
\begin{equation}\label{eq3}
\begin{cases}
u_t+u^Q\pa_xu=-\pa_x(1-\pa^2_x)^{-1}[\frac{Q^2+3Q}{2(Q+1)}u^{Q+1}+\frac Q2u^{Q-1}(\pa_xu)^2], \; &(t,x)\in \R^+\times\R,\\
u(0,x)=u_0,\; &x\in \R.
\end{cases}
\end{equation}
Setting $\Lambda^{-2}=(1-\pa^2_x)^{-1}$, then $\Lambda^{-2}f=G*f$ where $G(x)=\fr12e^{-|x|}$ is the kernel of the operator $\Lambda^{-2}$. We can transform the generalized Camassa-Holm equation into the following transport type equation
\begin{equation}\label{N}
\begin{cases}
u_t+u^Qu_x=\mathbf{P}_1(u)+\mathbf{P}_2(u),\\
u(x,t=0)=u_0(x),
\end{cases}
\end{equation}
where
\begin{equation}\label{6}
\mathbf{P}_1(u)=-\Lambda^{-2}(c_1u^{Q-2}u_x^3)
-\pa_x\Lambda^{-2}\left(c_2u^{Q+1}\right)\quad\text{and}\quad \mathbf{P}_2(u)=-\pa_x\Lambda^{-2}\left(c_3u^{Q-1}u^2_x\right),
\end{equation}
with $c_1=0,c_2=\frac{Q^2+3Q}{2(Q+1)}$ and $c_3=\frac Q2$. Our main result is as follows.
\begin{theorem}\label{th1}
Let $Q\geq 2$. The Cauchy problem for the generalized Camassa-Holm equation (\ref{eq1}) is ill-posed in $B^1_{\infty,1}(\R)$ in the sense of Hadamard. More precisely, for large enough $n\in \mathbb{Z}^+$, there exists an initial data $u_{0,n}$ such that the generalized Camassa-Holm equation (\ref{eq1}) has a solution $u_n\in \mathcal{C}([0,1];H^{3})$ satisfying
\bbal
\|u_{0,n}\|_{B^1_{\infty,1}}\leq \frac{1}{\log\log n}\quad\text{but}\quad
\|u_n(t_n)\|_{B^1_{\infty,1}}\geq {\log\log n},
\end{align*}
with $t_n\in \left(0,\frac{1}{\log n}\right]$.
\end{theorem}
\begin{remark}
In fact Theorem \ref{th1} shows the norm inflation in the space $B^1_{\infty,1}(\R)$, which implies the data-to-solution map of the equation (\ref{eq1}) is discontinuous at the origin $u_0=0$. In the sense of Hadamard, that means the Cauchy problem for the generalized Camassa-Holm equation (\ref{eq1}) is ill-posed in $B^1_{\infty,1}(\R)$.
\end{remark}

\section{Preliminaries}\label{sec2}
In this section, we introduce some basic definitions and lemmas related to Fourier transform, such as Littlewood-Paley decomposition, Besov spaces and Moser type estimation. Some common notations of this paper are given at the end.

Let us recall that for all $f\in \mathcal{S}$, the Fourier transform $\widehat{f}$, is defined by
$$
(\mathcal{F} f)(\xi)=\widehat{f}(\xi)=\int_{\R}e^{-ix\xi}f(x)\dd x \quad\text{for any}\; \xi\in\R.
$$
 The inverse Fourier transform of any $g$ is given by
$$
(\mathcal{F}^{-1} g)(x)=\check{g}(x)=\frac{1}{2 \pi} \int_{\R} g(\xi) e^{i x \cdot \xi} \dd \xi.
$$

Next, we will recall some facts about the Littlewood-Paley decomposition and the nonhomogeneous Besov spaces (see \cite{B} for more details).
Let $\mathcal{B}:=\{\xi\in\mathbb{R}:|\xi|\leq 4/3\}$ and $\mathcal{C}:=\{\xi\in\mathbb{R}:3/4\leq|\xi|\leq 8/3\}.$
Choose a radial, non-negative, smooth function $\chi:\R\mapsto [0,1]$ such that it is supported in $\mathcal{B}$ and $\chi\equiv1$ for $|\xi|\leq3/4$. Setting $\varphi(\xi):=\chi(\xi/2)-\chi(\xi)$, then we deduce that $\varphi$ is supported in $\mathcal{C}$. Moreover,
\begin{eqnarray*}
\chi(\xi)+\sum_{j\geq0}\varphi(2^{-j}\xi)=1 \quad \mbox{ for any } \xi\in \R.
\end{eqnarray*}
We should emphasize that the fact $\varphi(\xi)\equiv 1$ for $4/3\leq |\xi|\leq 3/2$ which will be used in the sequel.

For every $u\in \mathcal{S'}(\mathbb{R})$, the inhomogeneous dyadic blocks ${\Delta}_j$ are defined as follows
\begin{numcases}{\Delta_ju=}
0, & if $j\leq-2$;\nonumber\\
\chi(D)u=\mathcal{F}^{-1}(\chi \mathcal{F}u), & if $j=-1$;\nonumber\\
\varphi(2^{-j}D)u=\mathcal{F}^{-1}\f(\varphi(2^{-j}\cdot)\mathcal{F}u\g), & if $j\geq0$.\nonumber
\end{numcases}
In the inhomogeneous case, the following Littlewood-Paley decomposition makes sense
$$
u=\sum_{j\geq-1}{\Delta}_ju\quad \text{for any}\;u\in \mathcal{S'}(\mathbb{R}).
$$
\begin{definition}
  Let $s\in\mathbb{R}$ and $(p,r)\in[1, \infty]^2$. The nonhomogeneous Besov space $B^{s}_{p,r}(\R)$ is defined by
  \begin{align*}
  B^{s}_{p,r}(\R):=\Big\{f\in \mathcal{S}'(\R):\;\|f\|_{B^{s}_{p,r}(\mathbb{R})}<\infty\Big\},
  \end{align*}
  where
  \begin{numcases}{\|f\|_{B^{s}_{p,r}(\mathbb{R})}=}
  \left(\sum_{j\geq-1}2^{sjr}\|\Delta_jf\|^r_{L^p(\mathbb{R})}\right)^{\fr1r}, &if $1\leq r<\infty$,\nonumber\\
  \sup_{j\geq-1}2^{sj}\|\Delta_jf\|_{L^p(\mathbb{R})}, &if $r=\infty$.\nonumber
  \end{numcases}
  \end{definition}

We need the following inequality, which is a generalization of a result due to Bernstein.
\begin{lemma}[Lemma 2.1 in \cite{B}] \label{lem2.1} Let $\mathcal{B}$ be a Ball and $\mathcal{C}$ be an annulus. There exist constants $C>0$ such that for all $k\in \mathbb{N}\cup \{0\}$, any $\lambda\in \R^+$ and any function $f\in L^p$ with $1\leq p \leq q \leq \infty$, we have
\begin{align*}
&{\rm{supp}}\widehat{f}\subset \lambda \mathcal{B}\;\Rightarrow\; \|\pa_x^kf\|_{L^q}\leq C^{k+1}\lambda^{k+(\frac{1}{p}-\frac{1}{q})}\|f\|_{L^p},  \\
&{\rm{supp}}\widehat{f}\subset \lambda \mathcal{C}\;\Rightarrow\; C^{-k-1}\lambda^k\|f\|_{L^p} \leq \|\pa_x^kf\|_{L^p} \leq C^{k+1}\lambda^k\|f\|_{L^p}.
\end{align*}
\end{lemma}
\begin{lemma}[Lemma 2.100 in \cite{B}] \label{lem2.2}
Let $1 \leq r \leq \infty$, $1 \leq p \leq p_{1} \leq \infty$ and $\frac{1}{p_{2}}=\frac{1}{p}-\frac{1}{p_{1}}$. There exists a constant $C$ depending continuously on $p,p_1$, such that
$$
\left\|\left(2^{j}\left\|[\Delta_{j},v \pa_x] f\right\|_{L^{p}}\right)_{j}\right\|_{\ell^{r}} \leq C\left(\|\pa_x v\|_{L^{\infty}}\|f\|_{B_{p, r}^{1}}+\|\pa_x f\|_{L^{p_{2}}}\|\pa_x v\|_{B_{p_1,r}^{0}}\right).
$$
\end{lemma}
{\bf Notations :}
 $C$ stands for some positive constant independent of $n$, which may vary from line to line.
 The symbol $A\approx B$ means that $C^{-1}B\leq A\leq CB$.
 We shall call a ball $B(x_0,r)=\{x\in \R: |x-x_0|\leq R\}$ with $R>0$ and an annulus $\mathcal{C}(0,r_1,r_2)=\{x\in \R: 0<r_1\leq|x|\leq r_2\}$ with $0 <r_1 <r_2$.
  Given a Banach space $X$, we denote its norm by $\|\cdot\|_{X}$. We shall use the simplified notation $\|f,\cdots,g\|_X=\|f\|_X+\cdots+\|g\|_X$ if there is no ambiguity.
 We will also define the Lipschitz space $C^{0,1}$ using the norm $\|f\|_{C^{0,1}}=\|f\|_{L^\infty}+\|\pa_xf\|_{L^\infty}$.
 For $I\subset\R$, we denote by $\mathcal{C}(I;X)$ the set of continuous functions on $I$ with values in $X$. Sometimes we will denote $L^p(0,T;X)$ by $L_T^pX$.

\section{Construction of Initial Data}
Define a smooth cut-off function $\chi$ with values in $[0,1]$ which satisfies
\bbal
\chi(\xi)=
\bca
1, \quad \mathrm{if} \ |\xi|\leq \frac{1}{4^Q},\\
0, \quad \mathrm{if} \ |\xi|\geq \frac{1}{2^Q}.
\eca
\end{align*}
For simplicity, we set $\gamma:=\frac{17}{24}$ and dnote
\bbal
n\in 16\mathbb{N}=\left\{16,32,48,\cdots\right\}\quad\text{and}\quad
\mathbb{N}(n)=\left\{k\in 8\mathbb{N}: \frac{n}4 \leq k\leq \frac{n}2\right\}.
\end{align*}
We introduce the following new notation which will be used often throughout this paper
  \bbal
  \|f\|_{B^k_{\infty,1}\left(\mathbb{N}(n)\right)}=\sum_{j\in\mathbb{N}(n)}2^k\|\Delta_jf\|_{L^\infty},\quad k\in\{0,1\}.
  \end{align*}
Inspired by \cite{gyy,lyz3}, we define the initial data $u_0$ which contains high frequency and low frequency by
\bbal
u_{0,n}&=n^{-\frac{1}{Q+1}}\f(u^{\rm{H}}_{0,n}+u^{\rm{L}}_{0,n}\g),
\end{align*}
where
\bbal
u^{\rm{H}}_{0,n}&:=2^{-n}\log n\sum_{\ell\in \mathbb{N}(n)}\cos\left(2^n\gamma(x+2^{\ell+1}\gamma)\right)\cd
\cos\left(2^{\ell}\gamma(x+2^{\ell+1}\gamma)\right)\cd\check{\chi}(x+2^{\ell+1}\gamma),\\
u^{\rm{L}}_{0,n}&:=\sum_{\ell\in \mathbb{N}(n)}\check{\chi}(x+2^{\ell+1}\gamma).
\end{align*}
For any $1\leq M\leq Q$, it is easy to check that
\bal\label{hi-supp}
&\mathrm{supp} \ \mathcal{F}\f(\cos\left(2^n\gamma(x+2^{\ell+1}\gamma)\right)\cd\cos\left(2^{\ell}\gamma(x+2^{\ell+1}\gamma)\right)
\cd\check{\chi}^M(x+2^{\ell+1}\gamma)\g)\nonumber\\
&\qquad\subset \left\{\xi\in\R: \ 2^{n}\gamma-2^{\ell}\gamma-\fr12\leq |\xi|\leq 2^{n}\gamma+2^{\ell}\gamma+\fr12\right\},
\end{align}
which implies
\bal\label{supp}
\mathrm{supp} \ \widehat{u^{\rm{L}}_{0,n}}\subset \left\{\xi\in\R: \ |\xi|\leq \frac{1}{2^Q}\right\}, \qquad \mathrm{supp}\ \widehat{u^{\rm{H}}_{0,n}}\subset \left\{\xi\in\R: \ \fr43 2^{n-1}\leq |\xi|\leq \fr32 2^{n-1}\right\}.
\end{align}

\begin{lemma}\label{le-e1}
There exists a positive constant $C$ independent of $n$ such that
\bbal
2^{n}\|u^{\rm{H}}_{0,n}\|_{L^\infty}+\|\pa_xu^{\rm{H}}_{0,n}\|_{L^\infty}\leq C \log n,\qquad
\|u^{\rm{L}}_{0,n}\|_{C^{0,1}}\leq C.
\end{align*}
\end{lemma}

\begin{lemma}\label{le-e2}
There exists a positive constant $c$ independent of $n$ such that
\bbal
\left\|u^{Q-1}_{0,n}(\pa_xu_{0,n})^2\right\|_{B^0_{\infty,1}\left(\mathbb{N}(n)\right)}\geq c(\log n)^{2}, \qquad n\gg1.
\end{align*}
\end{lemma}
\begin{proof}
Notice that
\bbal
&\quad \ u^{Q-1}_{0,n}(\pa_xu_{0,n})^2
\\&=\underbrace{\frac1n(u^{\rm{L}}_{0,n})^{Q-1}(\pa_xu^{\rm{H}}_{0,n})^2}_{=:\;\mathbf{I}_1}
+\underbrace{\frac1n[u^{Q-1}_{0,n}-(u^{\rm{L}}_{0,n})^{Q-1}](\pa_xu^{\rm{H}}_{0,n}+\pa_xu^{\rm{L}}_{0,n})^2}_{=:\;\mathbf{I}_2}
\\& \quad +\underbrace{\frac1n(u^{\rm{L}}_{0,n})^{Q-1}\f((\pa_xu^{\rm{L}}_{0,n})^2+2\pa_xu^{\rm{H}}_{0,n}\pa_xu^{\rm{L}}_{0,n}\g)}_{=:\;\mathbf{I}_3}.
\end{align*}
For the term $\mathbf{I}_2$, we can deduce from Lemma \ref{le-e1} that
\bal\label{i2}
&\left\|\mathbf{I}_2\right\|_{B^0_{\infty,1}\left(\mathbb{N}(n)\right)}\leq Cn \|\mathbf{I}_2\|_{L^\infty}\leq C\|u^{\rm{H}}_{0,n}\|_{L^\infty}||u_{0,n}||^{Q-2}_{L^\infty}\|\pa_xu^{\rm{H}}_{0,n},\pa_xu^{\rm{L}}_{0,n}\|^2_{L^{\infty}}\leq C2^{-n}(\log n)^{Q+1}.
\end{align}
For the term $\mathbf{I}_3$, using \eqref{supp}, one has
\bal\label{i3}
&\De_j\mathbf{I}_3=0\quad\text{for}\; j\in \mathbb{N}(n)\quad\Rightarrow\quad
\left\|\mathbf{I}_3\right\|_{B^0_{\infty,1}\left(\mathbb{N}(n)\right)}=0.
\end{align}
For the term $\mathbf{I}_1$, noticing that
\bbal
\pa_xu^{\rm{H}}_{0,n}&=-\gamma\log n\sum_{\ell\in \mathbb{N}(n)}\sin\left(2^n\gamma(x+2^{\ell+1}\gamma)\right)\cd
\cos\left(2^{\ell}\gamma(x+2^{\ell+1}\gamma)\right)\cd\check{\chi}(x+2^{\ell+1}\gamma)\\
&\quad+2^{-n}\log n\sum_{\ell\in \mathbb{N}(n)}\cos\left(2^n\gamma(x+2^{\ell+1}\gamma)\right)\cd\pa_x\Big(
\cos\left(2^{\ell}\gamma(x+2^{\ell+1}\gamma)\right)\cd\check{\chi}(x+2^{\ell+1}\gamma)\Big),
\end{align*}
then we decompose $\mathbf{I}_1$ as follows
\bbal
\mathbf{I}_{1}=\frac{1}{n}(\log n)^{2}\f(\mathbf{I}_{11}+\mathbf{I}_{12}-\mathbf{I}_{13}\g),
\end{align*}
where
\bbal
\mathbf{I}_{11}&=\gamma^2(u^{\rm{L}}_{0,n})^{Q-1}\f(\sum_{\ell\in \mathbb{N}(n)}\sin\left(2^n\gamma(x+2^{\ell+1}\gamma)\right)\cd
\cos\left(2^{\ell}\gamma(x+2^{\ell+1}\gamma)\right)\cd\check{\chi}(x+2^{\ell+1}\gamma)\g)^2,
\\ \mathbf{I}_{12}&=2^{-2n}(u^{\rm{L}}_{0,n})^{Q-1}\f(\sum_{\ell\in \mathbb{N}(n)}\cos\left(2^n\gamma(x+2^{\ell+1}\gamma)\right)\cd\pa_x\Big(
\cos\left(2^{\ell}\gamma(x+2^{\ell+1}\gamma)\right)\cd\check{\chi}(x+2^{\ell+1}\gamma)\Big)\g)^2,
\\ \mathbf{I}_{13}&=2\gamma2^{-n}(u^{\rm{L}}_{0,n})^{Q-1}\sum_{\ell\in \mathbb{N}(n)}\sin\left(2^n\gamma(x+2^{\ell+1}\gamma)\right)\cd
\cos\left(2^{\ell}\gamma(x+2^{\ell+1}\gamma)\right)\cd\check{\chi}(x+2^{\ell+1}\gamma)\\
&\quad\times\sum_{\ell\in \mathbb{N}(n)}\cos\left(2^n\gamma(x+2^{\ell+1}\gamma)\right)\cd\pa_x\Big(
\cos\left(2^{\ell}\gamma(x+2^{\ell+1}\gamma)\right)\cd\check{\chi}(x+2^{\ell+1}\gamma)\Big).
\end{align*}
Easy computations give that
\bbal
\|\mathbf{I}_{12}\|_{L^\infty}&\leq C2^{-2n}\|u^{\rm{L}}_{0,n}\|^{Q-1}_{L^\infty}\left\|\sum_{\ell\in \mathbb{N}(n)}\pa_x\Big(
\cos\left(2^{\ell}\gamma(x+2^{\ell+1}\gamma)\right)\cd\check{\chi}(x+2^{\ell+1}\gamma)\Big)\right\|^2_{L^\infty}
\\&\leq C2^{-2n}\left\|\sum_{\ell\in \mathbb{N}(n)}\frac{2^{\ell}}{(1+|x+2^{\ell+1}\gamma|)^M}\right\|^2_{L^\infty}\leq C2^{-n},
\end{align*}
which implies
\bal\label{i12}
\left\|\mathbf{I}_{12}\right\|_{B^0_{\infty,1}\left(\mathbb{N}(n)\right)}\leq C n 2^{-n}.
\end{align}
Similarly, we have
\bal\label{i13}
\left\|\mathbf{I}_{13}\right\|_{B^0_{\infty,1}\left(\mathbb{N}(n)\right)}\leq C n 2^{-\frac{n}2}.
\end{align}
Using $\sin^2a\cos^2b=\fr14(1-\cos(2a))(1+\cos(2b))$, we can decompose $\mathbf{I}_{11}$ as
\bbal
\mathbf{I}_{11}=\gamma^2\sum_{i=1}^5\mathbf{I}_{11i},\quad\text{where}
\end{align*}
\bbal
\mathbf{I}_{111}&=\frac14(u^{\rm{L}}_{0,n})^{Q-1}\sum_{\ell\in \mathbb{N}(n)}\cos\big(2^{\ell+1}\gamma (x+2^{\ell+1}\gamma)\big)\check{\chi}^2(x+2^{\ell+1}\gamma) ,\\
\mathbf{I}_{112}&= \frac14(u^{\rm{L}}_{0,n})^{Q-1}\sum_{\ell\in \mathbb{N}(n)}\check{\chi}^2(x+2^{\ell+1}\gamma),
\\
\mathbf{I}_{113}&=-\frac14(u^{\rm{L}}_{0,n})^{Q-1}\sum_{\ell\in \mathbb{N}(n)}\cos\big(2^{n+1}\gamma (x+2^{\ell+1}\gamma)\big)\cd\check{\chi}^2(x+2^{\ell+1}\gamma),
\\
\mathbf{I}_{114}&= -\frac14(u^{\rm{L}}_{0,n})^{Q-1}\sum_{\ell\in \mathbb{N}(n)}\cos\big(2^{n+1}\gamma (x+2^{\ell+1}\gamma)\big)\cd\cos\big(2^{\ell+1}\gamma (x+2^{\ell+1}\gamma)\big)\cd\check{\chi}^2(x+2^{\ell+1}\gamma),
\\
\mathbf{I}_{115}&=(u^{\rm{L}}_{0,n})^{Q-1}\sum_{\ell,j\in \mathbb{N}(n)\atop\ell\neq j}\Big(\sin\left(2^n\gamma(x+2^{\ell+1}\gamma)\right)\cd
\cos\left(2^{\ell}\gamma(x+2^{\ell+1}\gamma)\right)\cd\check{\chi}(x+2^{\ell+1}\gamma)
  \\& \quad  \times \sin\left(2^n\gamma(x+2^{j+1}\gamma)\right)\cd
\cos\left(2^{j}\gamma(x+2^{j+1}\gamma)\right)\cd\check{\chi}(x+2^{j+1}\gamma)\Big).
\end{align*}
By \eqref{hi-supp}, we have
\bal\label{i112}
\De_j\mathbf{I}_{112}=\De_j\mathbf{I}_{113}=\De_j\mathbf{I}_{114}=0\quad\text{for}\; j\in \mathbb{N}(n) \quad \Rightarrow \quad
\|\mathbf{I}_{112},\mathbf{I}_{113},\mathbf{I}_{114}\|_{B^0_{\infty,1}\left(\mathbb{N}(n)\right)}=0.
\end{align}
Using Lemma \ref{le-e1}, we have
\bal\label{i115}
\|\mathbf{I}_{115}\|_{B^0_{\infty,1}(\mathbb{N}(n))}&\leq Cn\|\mathbf{I}_{115}\|_{L^\infty}
\leq Cn\|u^{\rm{L}}_{0,n}\|^{Q-1}_{L^\infty}\sum_{\ell,j\in \mathbb{N}(n)\atop
\ell\neq j}\left\|\check{\chi}(x+2^{\ell+1}\gamma)
  \cd\check{\chi}(x+2^{j+1}\gamma)\right\|_{L^\infty} \nonumber
\\&\leq Cn\sum_{j>\ell\in \mathbb{N}(n)}\left\|(1+|x+2^{j+1}\gamma|^2)^{-M}
(1+|x+2^{\ell+1}\gamma|^2)^{-M}\right\|_{L^\infty} \nonumber \\
&\leq Cn\sum_{j>\ell\in \mathbb{N}(n)}\left\|(1+|x|^2)^{-M}(1+|x-(2^{j+1}-2^{\ell+1})\gamma|^2))^{-M}\right\|_{L^\infty} \nonumber
\\&\leq Cn\sum_{j>\ell\in \mathbb{N}(n)}\left(\gamma (2^{j}-2^{\ell})\right)^{-2M}
\leq Cn^32^{-\frac{nM}{2}}.
\end{align}
Finally, we can break $\mathbf{I}_{111}$ down into three parts
\bbal
\mathbf{I}_{111}&=\frac{1}{4}\sum_{\ell\in \mathbb{N}(n)}\cos\big(2^{\ell+1}\gamma (x+2^{\ell+1}\gamma)\big)\check{\chi}^{Q+1}(x+2^{\ell+1}\gamma)\\
&\quad+\frac{1}{4}\sum_{\ell,j\in \mathbb{N}(n)\atop
\ell\neq j}\cos\big(2^{\ell+1}\gamma (x+2^{\ell+1}\gamma)\big)\check{\chi}^2(x+2^{\ell+1}\gamma)
\check{\chi}^{Q-1}(x+2^{j+1}\gamma)
\\
&\quad+\frac{1}{4}\sum_{\ell\in \mathbb{N}(n)}\cos\big(2^{\ell+1}\gamma (x+2^{\ell+1}\gamma)\big)\check{\chi}^2(x+2^{\ell+1}\gamma)\cdot\sum_{\ell_1,\cdots,\ell_{Q-1}\in \mathbb{N}(n)\atop
\exists\ell_i\neq \ell_j, \ 1\leq i \neq j \leq Q-1}\check{\chi}(x+2^{\ell_1+1}\gamma)\cdots \check{\chi}(x+2^{\ell_{Q-1}+1}\gamma)\\
&:=\frac{1}{4}\mathbf{I}_{1111}+\frac{1}{4}\mathbf{I}_{1112}+\frac{1}{4}\mathbf{I}_{1113}.
\end{align*}
Due to \eqref{hi-supp}, we have
\begin{equation*}
{\dot{\Delta}_j\mathbf{I}_{1111}=\mathcal{F}^{-1}\left(\varphi(2^{-j}\cdot)\mathcal{F}\mathbf{I}_{1111}\right)=}
\begin{cases}
\cos\big(2^{j+1}\gamma (x+2^{j+1}\gamma)\big)\check{\chi}^{Q+1}(x+2^{j+1}\gamma), &\text{if}\; \ell=j,\\
0, &\text{if}\; \ell\neq j,
\end{cases}
\end{equation*}
which implies
\bal\label{i1111}
\|\mathbf{I}_{1111}\|_{B^0_{\infty,1}(\mathbb{N}(n))}&=\sum_{j\in \mathbb{N}(n)}\left\|\cos\big(2^{j+1}\gamma (x+2^{j+1}\gamma)\big)\check{\chi}^{Q+1}(x+2^{j+1}\gamma)\right\|_{L^\infty}\geq \sum_{j\in \mathbb{N}(n)}\check{\chi}^{Q+1}(0)\geq cn.
\end{align}
Following the same procedure as $\mathbf{I}_{115}$, we get
\bal\label{i1112}
\|\mathbf{I}_{1112},\mathbf{I}_{1113}\|_{B^0_{\infty,1}(\mathbb{N}(n))}\leq  Cn^{Q}2^{-\frac{nM}{2}},
\end{align}
Combining the above estimates \eqref{i2}-\eqref{i1112}, we obtain that for large enough $n$
\bbal
&\quad \|u^{Q-1}_{0,n}(\pa_xu_{0,n})^2\|_{B^0_{\infty,1}(\mathbb{N}(n))}
\\&\geq \|\mathbf{I}_{1}\|_{B^0_{\infty,1}(\mathbb{N}(n))}-\|\mathbf{I}_{2}\|_{B^0_{\infty,1}(\mathbb{N}(n))}\\
&\geq c\frac{1}{n}(\log n)^{2}\f(\|\mathbf{I}_{1111}\|_{B^0_{\infty,1}(\mathbb{N}(n))}-
\|\mathbf{I}_{12},\mathbf{I}_{13},\mathbf{I}_{115},\mathbf{I}_{1112},\mathbf{I}_{1113}\|_{B^0_{\infty,1}(\mathbb{N}(n))}\g)-C2^{-n}(\log n)^{Q+1}\\
&\geq c(\log n)^{2}.
\end{align*}
This completes the proof of Lemma \ref{le-e2}.
\end{proof}

\section{Proof of the main theorem}

By classical result, we can obtian that the generalized Camassa-Holm equation has a solution $u_n\in \mathcal{C}([0,1];H^{3})$ with the initial data $u_{0,n}$. Let $\phi_n$ satisfy the following ODE:
\begin{align}\label{ode}
\quad\begin{cases}
\frac{\dd}{\dd t}\phi_n(t,x)=u^{Q}_n(t,\phi_n(t,x)),\\
\phi_n(0,x)=x,
\end{cases}
\end{align}
which is equivalent to
\bal\label{n}
\phi_n(t,x)=x+\int^t_0u^{Q}_n(\tau,\phi_n(\tau,x))\dd \tau.
\end{align}
Considering the transport equation
\begin{align}\label{pde}
\quad\begin{cases}
\pa_tv+u^{Q}_n\pa_xv=P,\\
v(0,x)=v_0(x),
\end{cases}
\end{align}
we get from \eqref{pde} that
\bbal
\pa_t(\De_jv)+u^{Q}_n\pa_x\De_jv&=R_j+\Delta_jP,
\end{align*}
with $R_j=[u^Q_n,\De_j]\pa_xv=u^Q_n\De_j\pa_xv-\Delta_j(u^Q_n\pa_xv)$. Due to \eqref{ode}, then
\bbal
\frac{\dd}{\dd t}\left((\De_jv)\circ\phi_n\right)&=R_j\circ\phi_n+\Delta_jP\circ\phi_n,
\end{align*}
which means that
\bal\label{l6}
\De_jv\circ\phi_n=\De_jv_0+\int^t_0R_j\circ\phi_n\dd \tau+\int^t_0\Delta_jP\circ\phi_n\dd \tau.
\end{align}
For $n\gg1$, we have for $t\in[0,1]$
\bbal
\|u_n\|_{C^{0,1}}\leq C\|u_{0,n}\|_{C^{0,1}}\leq C n^{-\frac{1}{Q+1}}\log  n.
\end{align*}
To prove Theorem \ref{th1}, it suffices to show that there exists $t_0\in(0,\frac{1}{\log n}]$ such that
\bal\label{holds}
\|u(t_0,\cdot)\|_{B^1_{\infty,1}}\geq \log\log n.
\end{align}
If \eqref{holds} were not true, then
\bal\label{nholds}
\sup\limits_{t\in(0,\frac{1}{\log n}]}\|u(t,\cdot)\|_{B^1_{\infty,1}}< \log\log n.
\end{align}
Utilizing \eqref{l6} to \eqref{N} yields
\bbal
(\De_ju_n)\circ \phi_n&=\De_ju_{0,n}+\int^t_0R^1_{j,n}\circ \phi_n\dd \tau +\int^t_0\De_jF_n\circ \phi_n\dd \tau
\\& \qquad +c_3\int^t_0\big(\De_jE_n\circ \phi_n-\De_jE_{0,n}\big)\dd \tau+c_3t\De_jE_{0,n},
\end{align*}
where
\bbal
&R^1_{j,n}=[u^{Q}_n,\De_j]\pa_xu_n,\qquad
F_n=-\Lambda^{-2}\f(c_1u^{Q-2}_n(\pa_xu_n)^3+c_2\pa_x(u^{Q+1}_n)\g), \\
&E_n=-\pa_x\Lambda^{-2}\f(u^{Q-1}_n(\pa_xu_n)^2\g), \qquad E_{0,n}=-\pa_x\Lambda^{-2}\f(u^{Q-1}_{0,n}(\pa_xu_{0,n})^2\g).
\end{align*}
Due to Lemma \ref{le-e2}, we deduce
\bal\label{g1}
\sum_{j\in \mathbb{N}(n)}2^j\|\De_jE_{0,n}\|_{L^\infty}
\approx \sum_{j\in \mathbb{N}(n)}\|\De_j\pa_xE_{0,n}\|_{L^\infty}\geq c\sum_{j\in \mathbb{N}(n)}\f\|\De_j[u^{Q-1}_{0,n}(\pa_xu_{0,n})^2]\g\|_{L^\infty}\geq c(\log n)^2.
\end{align}
Then, using the fact $\|f(t,\phi_n(t,x))\|_{L^\infty}= \|f(t,\cdot)\|_{L^\infty}$ and Lemma \ref{lem2.2}, we have
\bal\label{g2}
\sum_{j\geq -1}2^j\|R^1_{j,n}\circ \phi_n\|_{L^\infty}&=\sum_{j\geq -1}2^j\|R^1_{j,n}\|_{L^\infty} \nonumber
\\&\leq C\|\pa_x(u^{Q}_n)\|_{B^0_{\infty,1}}\|u_n\|_{B^1_{\infty,1}}
\\&\leq C\|u_n\|^{Q-1}_{C^{0,1}}\|u_n\|^2_{B^1_{\infty,1}}\leq C  n^{-\frac{Q-1}{Q+1}}(\log  n)^{Q+1}. \nonumber
\end{align}
Also, we have
\bal\label{g3}
\sum_{j\in \mathbb{N}(n)}2^j\|\De_jF_n\circ \phi_n\|_{L^\infty}&\leq C\sum_{j\in \mathbb{N}(n)}2^j\|\De_jF_n\|_{L^\infty}
\nonumber\\
&\leq C\|u^{Q-2}_n(\pa_xu_n)^3+\pa_x(u^{Q+1}_n)\|_{L^\infty}\nonumber\\
&\leq C\|u_n\|^{Q+1}_{C^{0,1}}
\leq C\|u_{0,n}\|^{Q+1}_{C^{0,1}}\leq Cn^{-1}(\log n)^{Q+1}.
\end{align}
Combining \eqref{g1}-\eqref{g3} and using Lemmas \ref{le-e1}-\ref{le-e2} yields
\bal\label{con1}
\sum_{j\in \mathbb{N}(n)}2^j\|(\De_ju_n)\circ \phi_n\|_{L^\infty}
&\geq t\sum_{j\in \mathbb{N}(n)}2^j\|\De_jE_{0,n}\|_{L^\infty}-\sum_{j\in \mathbb{N}(n)}2^j\|\De_jE_n\circ \phi_n-\De_jE_{0,n}\|_{L^\infty} \nonumber\\
&\quad-C  n^{-\frac{Q-1}{Q+1}}(\log  n)^{Q+1}
-C\|u_{0,n}\|_{B^1_{\infty,1}} \nonumber
\\&\geq ct\log^2n-\sum_{j\in \mathbb{N}(n)}2^j\|\De_jE_n\circ \phi_n-\De_jE_{0,n}\|_{L^\infty}-C n^{-\frac{1}{Q+1}}(\log  n)^{Q+1}.
\end{align}
Next, we need to estimate the term $\De_jE_n\circ \phi_n-\De_jE_{0,n}$. Using
\bbal
&\quad \pa_x\Lambda^{-2}[(Q-1)u^{2Q-2}_n(\pa_xu_n)^3+2u^{Q-1}_n\pa_x(u^Q_n\pa_xu_n)\pa_xu_n]
\\&=\pa_x\Lambda^{-2}[Qu^{2Q-2}_n(\pa_xu_n)^3+\pa_x(u^{2Q-1}_n(\pa_xu_n)^2)]
\\&=Q\pa_x\Lambda^{-2}(u^{2Q-2}_n(\pa_xu_n)^3)+\Lambda^{-2}(u^{2Q-1}_n(\pa_xu_n)^2)-u^{2Q-1}_n(\pa_xu_n)^2,
\end{align*}
then we find that
\bal\label{E}
&\quad \pa_tE_n+u^{Q}_n\pa_xE_n \nonumber
\\&=-\pa_x\Lambda^{-2}\pa_t(u^{Q-1}_n(\pa_xu_n)^2)
-u^Q_n\pa_x^2\Lambda^{-2}(u^{Q-1}_n(\pa_xu_n)^2)\nonumber\\
&=-\pa_x\Lambda^{-2}((Q-1)u^{Q-2}_n\pa_tu_n (\pa_xu_n)^2+2u^{Q-1}_n\pa_xu_n\pa_t\pa_xu_n)+u^{2Q-1}_n(\pa_xu_n)^2-u^Q_n\Lambda^{-2}(u^{Q-1}_n(\pa_xu_n)^2)\nonumber\\
&=\mathrm{J}_n+\pa_x\Lambda^{-2}[(Q-1)u^{2Q-2}_n(\pa_xu_n)^3+2u^{Q-1}_n\pa_x(u^{Q}_n\pa_xu_n)\pa_xu_n]+u^{2Q-1}_n(\pa_xu_n)^2-u^Q_n\Lambda^{-2}(u^{Q-1}_n(\pa_xu_n)^2)\nonumber\\
&=\mathrm{J}_n+\mathrm{K}_n,
\end{align}
where
\bbal
&\mathrm{J}_n=-\pa_x\Lambda^{-2}\f((Q-1)u^{Q-2}_n[\mathbf{P}_1(u_n)+\mathbf{P}_2(u_n)](\pa_xu_n)^2
+2\pa_x[\mathbf{P}_1(u_n)+\mathbf{P}_2(u_n)]u^{Q-1}_n\pa_xu_n
\g),\\
&\mathrm{K}_n=Q\pa_x\Lambda^{-2}(u^{2Q-2}_n(\pa_xu_n)^3)
+\Lambda^{-2}(u^{2Q-1}_n(\pa_xu_n)^2)
-u^Q_n\Lambda^{-2}(u^{Q-1}_n(\pa_xu_n)^2).
\end{align*}
Utilizing \eqref{l6} to \eqref{E} yields
\bbal
\De_jE_n\circ \phi_n-\De_jE_{0,n}=\int^t_0[u^Q_n,\De_j]\pa_xE_n\circ \phi_n\dd \tau +\int^t_0\De_j(\mathrm{J}_n+\mathrm{K}_n)\circ \phi_n\dd \tau.
\end{align*}
Using the commutator estimate from Lemma \ref{lem2.2}, one has
\bbal
2^j\|[u^Q_n,\De_j]\pa_xE_n\|_{L^\infty}&\leq C(\|\pa_x(u^Q_n)\|_{L^\infty}\|E_n\|_{B^1_{\infty,\infty}}
+\|\pa_xE_n\|_{L^\infty}\|u_n^Q\|_{B^1_{\infty,\infty}})\\
&\leq C\|u_n\|^{2Q+1}_{C^{0,1}}\leq Cn^{-\frac{2Q+1}{Q+1}}(\log n)^{2Q+1}.
\end{align*}
Due to the facts
\bbal
\|\Lambda^{-2}f\|_{L^\infty}+ \|\pa_x\Lambda^{-2}f\|_{L^\infty}\leq 2\|f\|_{L^\infty}\quad \Rightarrow\quad \|\pa^2_x\Lambda^{-2}f\|_{L^\infty}\leq 2\|f\|_{L^\infty},
\end{align*}
then we have
\bbal
2^j\|\De_j\mathrm{J}_n\|_{L^\infty}\approx\|\pa_x\mathrm{J}_n\|_{L^\infty}\leq C \|u_n\|^{2Q+1}_{C^{0,1}}\leq Cn^{-\frac{{2Q+1}}{Q+1}}(\log n)^{2Q+1}.
\end{align*}
Similarly,
\bbal
2^j\|\De_j\mathrm{K}_n\|_{L^\infty}\leq C \|u_n\|^{2Q+1}_{C^{0,1}}\leq Cn^{-\frac{{2Q+1}}{Q+1}}(\log n)^{2Q+1}.
\end{align*}
Then, we deduce that
\bbal
2^j\|\De_jE_n\circ \phi_n-\De_jE_{0,n}\|_{L^\infty}\leq C \|u_n\|^{2Q+1}_{C^{0,1}}\leq Cn^{-\frac{{2Q+1}}{Q+1}}(\log n)^{2Q+1},
\end{align*}
which leads to
\bal\label{E-E0}
\sum_{j\in \mathbb{N}(n)}2^j\|\De_jE_n\circ \phi_n-\De_jE_{0,n}\|_{L^\infty}
\leq Cn^{-\frac{{Q}}{Q+1}}(\log n)^{2Q+1}.
\end{align}
Combining \eqref{con1} and \eqref{E-E0}, then for $t=\frac{1}{\log n}$, we obtain for $n\gg1$
\bbal
\|u_n(t)\|_{B^1_{\infty,1}}&\geq \|u_n(t)\|_{B^1_{\infty,1}(\mathbb{N}(n))}\\
&\geq c\sum_{j\in \mathbb{N}(n)}2^j\|(\De_ju)\circ \phi\|_{L^\infty}
\\&\geq c t(\log n)^2-Cn^{-\frac{{Q}}{Q+1}}(\log n)^{2Q+1}-C  n^{-\frac{1}{Q+1}}(\log  n)^{Q+1}\\
&\geq \log\log n,
\end{align*}
which contradicts the hypothesis \eqref{nholds}. Thus, Theorem \ref{th1} is proved.{\hfill $\square$}

\section*{Acknowledgments}
 M. Li is supported by Natural Science Foundation of Jiangxi Province (20212BAB211011). W. Zhu is supported by the National Natural Science Foundation of China (12201118) and Guangdong Basic and Applied Basic Research Foundation (2021A1515111018).

\section*{Data Availability} No data was used for the research described in the article.

\section*{Conflict of interest}
The authors declare that they have no conflict of interest.


\addcontentsline{toc}{section}{References}
\begin{thebibliography}{99}
\linespread{0}\addtolength{\itemsep}{-1.0ex}

\bibitem{B} H. Bahouri, J. Y. Chemin, R. Danchin, Fourier Analysis and Nonlinear Partial Differential Equations, Grundlehren der Mathematischen Wissenschaften, Springer, Heidelberg, 2011.
\bibitem{bc1}A. Bressan and A. Constantin, \textit{Global conservative solutions of the Camassa-Holm equation}, {Arch. Ration. Mech. Anal.}, {\bf183} (2007), 215-239.
\bibitem{bc2} A. Bressan and A. Constantin, \textit{Global dissipative solutions of the Camassa-Holm equation}, {Anal. Appl.}, {\bf5} (2007), 1-27.
\bibitem{ch}R. Camassa and D. D. Holm, \textit{An integrable shallow water equation with peaked solitons}, {Phys. Rev. Lett.}, {\bf71} (1993), 1661-1664.
\bibitem{cli} D. Chae and J. Liu, \textit{Blow-up, zero $\alpha$ limit and the Liouville type theorem for the Euler-Poincar\'{e} equations}, {Comm. Math. Phys.}, {\bf314} (2012), 671-687.
\bibitem{c2} A. Constantin, \textit{Existence of permanent and breaking waves for a shallow water equation: a geometric approach}, {Ann. Inst. Fourier (Grenoble)}, {\bf50} (2000), 321-362.
\bibitem{c5} A. Constantin, \textit{The trajectories of particles in Stokes waves}, {Invent. Math.}, {\bf166} (2006), 523-535.
\bibitem{ce2} A. Constantin and J. Escher, \textit{Well-posedness, global existence, and blowup phenomena for a periodic quasi-linear hyperbolic equation}, {Comm. Pure Appl. Math.}, {\bf51} (1998), 475-504.
\bibitem{ce3} A. Constantin and J. Escher, \textit{Wave breaking for nonlinear nonlocal shallow water equations}, {Acta Math.}, {\bf181} (1998), 229-243.
\bibitem{ce4} A. Constantin and J. Escher, \textit{Global existence and blow-up for a shallow water equation}, {Ann. Scuola Norm. Sup. Pisa Cl. Sci. (4)},  {\bf26} (1998), 303-328.
\bibitem{d2} R. Danchin, \textit{A note on well-posedness for Camassa-Holm equation}, {J. Differ. Equ.}, {\bf192} (2003), 429-444.
\bibitem{ff} B. Fuchssteiner and A. S. Fokas,  \textit{Symplectic structures, their B?cklund transformations and hereditary symmetries}, {Phys. D}, {\bf 4(1)} (1981),47-66.
\bibitem{glmy} Z. Guo, X. Liu, M. Luc and Z. Yin, \textit{Ill-posedness of the Camassa-Holm and related equations in the critical space}, {J. Differ. Equ.}, {\bf266} (2019), 1698-1707.
\bibitem{gyy} Y. Guo, W. Ye and Z. Yin, \textit{Ill-posedness for the Cauchy problem of the Camassa-Holm equation in $B^1_{\infty,1}(\R)$}, {J. Differ. Equ.}, {\bf 327} (2022), 127-144.
\bibitem{hk1} S. Hakkaev and K. Kirchev, \textit{Local well-posedness and orbital stability of solitary wave solutions for the generalized Camassa-Holm equation}, {Commun. Part. Diff. Eq.}, {\bf 30} (2005),761-781.
\bibitem{hk2} S. Hakkaev and K. Kirchev, \textit{On the Well-posedness and Stability of Peakons for a Generalized Camassa-Holm Equation}, {International Journal of Nonlinear Science}, {\bf 1} (2006),139-148.
\bibitem{hk} A. Himonas and C. Kenig, \textit{Non-uniform dependence on initial data for the CH equation on the line}, {Diff. Integral Eqns}, {\bf22} (2009), 201--224.
\bibitem{hkm} A. Himonas, C. Kenig and Misio{\l}ek, \textit{Non-uniform dependence for the periodic CH equation}, {Commun. Part. Diff. Eqns}, {\bf35} (2010), 1145--1162.
\bibitem{lyz1} J. Li, Y. Yu and W. Zhu, \textit{Non-uniform dependence on initial data for the Camassa-Holm equation in Besov spaces}, {J. Differ. Equ.}, {\bf269} (2020), 8686--8700.
\bibitem{lyz2} J. Li, Y. Yu and W. Zhu, \textit{Ill-posedness for the Camassa-Holm and related equations in Besov spaces}, {J. Differ. Equ.}, {\bf306} (2022), 403--417.
\bibitem{lyz3} J. Li, Y. Yu and W. Zhu, \textit{Ill-posedness of the Novikov equation in the critical Besov space $B^1_{\infty,1}(\R)$}, arXiv:2210.02677.
\bibitem{lyz4} J. Li, Y. Yu and W. Zhu, \textit{Well-posedness and continuity properties of the Degasperis-Procesi equation in critical Besov space}, {Monatsh. Math.}, (2022) doi.org/10.1007/s00605-022-01691-4
\bibitem{lwyz} J. Li, X. Wu, Y. Yu and W. Zhu , \textit{Non-uniform dependence on initial data for the Camassa-Holm equation in the critical Besov space}, {J. Math. Fluid Mech.}, {\bf23} (2021), 1422-6928.
\bibitem{liy} J. Li and Z. Yin, \textit{Remarks on the well-posedness of Camassa-Holm type equations in Besov spaces}, {J. Differ. Equ.}, {\bf261} (2016), 6125-6143.
\bibitem{t} J. F. Toland, \textit{Stokes waves}, {Topol. Methods Nonlinear Anal.}, {\bf7} (1996), 1-48.
\bibitem{wyx} X. Wu, Y. Yu and Y. Xiao \textit{Non-uniform Dependence on Initial Data for the Generalized Camassa-Holm-Novikov Equation in Besov Space}, {J. Math. Fluid Mech.}, (2021) 23:104.
\bibitem{xz1} Z. Xin and P. Zhang, \textit{On the weak solutions to a shallow water equation}, {Comm. Pure Appl. Math.}, {\bf53} (2000), 1411-1433.
\bibitem{ylz} W. Yan, Y. Li and Y. Zhang, \textit{The Cauchy problem for the generalized Camassa-Holm equation in Besov space}, {J. Differ. Equ.}, {\bf 256} (2014), 2876-2901.
\bibitem{yyg}W. Ye, Z. Yin and Y. Guo. \textit{A new result for the local well-posedness of the Camassa-Holm type equations in critial Besov spaces $B^{1+1/p}_{p,1},1\leq p<\infty$}, arXiv: 2101.00803,202.

\end{thebibliography}
\end{document}